\documentclass[11pt,twoside]{amsart}

\numberwithin{equation}{section}
\usepackage{latexsym,amssymb,amsmath,amsthm,amsfonts,amscd}
\usepackage[backref=page]{hyperref}
\usepackage{pgfplots,enumerate,multicol,multirow,tabularx}
\usepackage{graphicx}
\usepackage{tikz}
\usetikzlibrary{matrix,arrows}
\tikzset{
  treenode/.style = {shape=rectangle, rounded corners,
                     draw, align=center,
                     top color=white,
                     bottom color=blue!10},
  root/.style     = {treenode, font=\Large,
                     bottom color=blue!30},
  env/.style      = {treenode, font=\ttfamily\normalsize},
  dummy/.style    = {circle,draw}
}

\usepackage{listings}
\definecolor{VerdeOlivo}{rgb}{0.3,0.5,0.1}
\definecolor{Magenta}{rgb}{.65,0.15,.2}
\definecolor{Gris}{gray}{0.3}

\newtheorem{Theorem}{Theorem}[section] 
\newtheorem{Definition}[Theorem]{Definition}
  
\newtheorem{Lemma}[Theorem]{Lemma} 

\newtheorem{Remark}[Theorem]{Remark}
\newtheorem{Example}[Theorem]{Example}

\newtheorem{Algorithm}[Theorem]{Algorithm}

\theoremstyle{definition}
\newcommand{\bff}[1]{{\bf #1}}

\begin{document}
	
	
\title[Arithmetical structures on dominated polynomials]{Arithmetical structures on dominated polynomials}
	
	
\author{Carlos E. Valencia}
\email[C. E. ~Valencia]{cvalencia@math.cinvestav.edu.mx, cvalencia75@gmail.com}
\author{Ralihe R. Villagr\'an}
\email[R. R. ~Villagr\'{a}n]{rvillagran@math.cinvestav.mx, ralihemath@gmail.com }
\thanks{Carlos E. Valencia was partially supported by
SNI and Ralihe R. Villagr\'an by CONACYT}
\address{
Departamento de Matem\'aticas\\
Centro de Investigaci\'on y de Estudios Avanzados del IPN\\
Apartado Postal 14--740 \\
07000 Mexico City, D.F. 
} 
	
	
\maketitle

\begin{abstract} 
In~\cite{algorithmic} was given an algorithm that computes arithmetical structures on matrices.
We use some of the ideas contained there to get an algorithm that computes arithmetical structures over dominated polynomials.
A dominated polynomial is an integer multivariate polynomial such that contains a monomial which is divided by all its monomials.
\end{abstract}



	
\section{Introduction}\label{intro}
Arithmetical structures on matrices was introduced in 2018 by Corrales and Valencia in~\cite{arithmetical}.
Were was proved that arithmetical structures on irreducible matrices are finite.
Recently arithmetical structures aroused some interest, see for instance~\cite{ PathsCycles, connectivityone, double, bident, extremal}.
In~\cite{algorithmic} was discussed some algorithmic aspects of arithmetical structures on matrices.

The main goal of this article is to generalizes the concept of arithmetical structure in the context of dominated polynomials and we get an algorithm to computes arithmetical structures on dominated polynomials. 
A dominated polynomial is an integer multivariate polynomial such that contains a monomial which is divided by all its monomials.
To finish we give an example how works the algorithm on a polynomial that is not the determinant of an integer matrix. 

We recall what it means of an arithmetical structure on a matrix. 
Given a non-negative integer matrix $L$ with zero diagonal (for instance the adjacency matrix of a graph), a pair $(\mathbf{d},\mathbf{r})\in \mathbb{N}_+^n\times \mathbb{N}_+^n$ is called an arithmetical structure of $L$ if 
\[
(\textrm{Diag}(\mathbf{d})-L)\mathbf{r}^t=\mathbf{0}^t\text{  and }\gcd(r_1,\ldots,r_n)=1.
\]


It is not difficult to check that the vector $\bff{d}$ 
is a solution of the polynomial Diophantine equation
\[
f_L(X):=\det(\textrm{Diag}(\mathbf{X})-L)=0.
\]

Therefore computing arithmetical structures on matrices consist on computing a subset of the solutions of a very special class of Diophantine equations, those whose polynomial is the determinant of a matrix with variables in the diagonal.

Throughout this article we use the usual partial order over $\mathbb{R}^n$ given by $\mathbf{a}\leq\mathbf{b}$ if and only if $a_i\leq b_i$ for all $i=1,\ldots,n$ and $\mathbf{a}$,$\mathbf{b}\in\mathbb{R}^n$. 
In a similar way, $\mathbf{a}<\mathbf{b}$ if and only if $\mathbf{a}\leq\mathbf{b}$ and $\mathbf{a}\neq\mathbf{b}$.
It is well known that this is a well partial order over $\mathbb{N}^n$. 


\section{Arithmetical structures on dominated polynomials.}\label{poly}
We begin by defining what is a dominated polynomial.
After that, we generalize Algorithms~\cite[3.2 and 3.4]{algorithmic} given for polynomials which are the determinant of a matrix with variables in the diagonal to dominated polynomials.
Some concepts are preserved in this new setting and others are not.
For instance, the concept of $d$-arithmetical structure is generalized easily.
However, this not happen in the case of the $r$-arithmetical structure.


\subsection{Dominated polynomials}
Given a polynomial $f\in \mathbb{Z}[\bff{x}]$, let $\mathcal{M}_f$ be its set of monomials with non zero coefficient.
A monomial $p\in\mathcal{M}_f$ is called dominant whenever is divided by every monomial in $\mathcal{M}_f$.

\begin{Definition}\label{dominated}
If $\mathcal{M}_f$ has a dominant monomial, then $f$ is called dominated.
\end{Definition} 

It is not difficult to check that if $\mathcal{M}_f$ has a dominant monomial, then it is unique.
Thus, let $p_f$ be the dominant monomial of a dominated polynomial $f$. 
We are interested when $f$ is square-free and $x_1\cdots x_n$ is its dominant monomial. 

Now we are ready to define a $d$-arithmetical structure of an irreducible square-free dominated polynomial. 

\begin{Definition}
Given a polynomial $f$ with its leading coefficient positive, an arithmetical structure of $f$ is a vector $\bff{d} \in \mathbb{N}_+^n$ such that $f(\bff{d})=0$ and all the non-constant coefficients of $f_{\bff{d}}(X):=f(X+\bff{d})$ are positive.
\end{Definition}

Note that if $f$ does not have its leading coefficient positive, then it does not have any arithmetical structures.
However, since either $f$ or $-f$ has its leading coefficient positive, then we can assume that $f$ has positive leading coefficient.
From here on out, let us assume that the leading coefficient is always positive unless the contrary is stated.

If a dominated square-free polynomial $f$ is reducible, that is, $f=\prod_{i=1}^sf_i$ for some irreducible square-free polynomials $f_i$, then each $f_i$ is a dominated polynomial.
Moreover, if $\bff{d}(f_i)$ is the vector with the entries of $\bff{d}$ that corresponds to the variables of $f_i$, then $\bff{d}$ is an arithmetical structure of $f$ if and only if $\bff{d}(f_i)$ is an arithmetical structure of at least one of the $f_i$ and the non-constant coefficients of $f_{i,{\bff{d}(f_i)}}(X)$ are positive and the constant coefficient is non-negative for all $i$.
Thus, if $f$ is reducible square-free polynomial, then it has an infinite number of arithmetical structures.

\begin{Definition}
Given a square-free dominated polynomial $f$ on $n$ variables, let
\[
\mathcal{D}(f)=\{ \bff{d}\in\mathbb{N}^n_+\, | \, \bff{d}\text{ is an arithmetical structure of } f\}.
\]
\end{Definition}

This definition generalizes the one given in~\cite[Section 2]{algorithmic}. 
More precisely, if $L$ is a non-negative matrix with zero diagonal, then $\mathcal{D}(L)=\mathcal{D}(f_L)$ where $f_L=\mathrm{det}(\mathrm{Diag}(\bff{x})-L)$.

Now, let $\mathcal{D}_{\geq 0}(f)=\Big\{ \mathbf{d}\in \mathbb{N}^n_+ \Big| \text{ all non-constant coefficients of }f_{\bff{d}}(X) \text{ positive and }f(\mathbf{d})\geq 0  \Big\}$.

\subsection{The algorithm for the dominated polynomial case}\label{polycase}
We extend Algorithms~\cite[3.2 and 3.4]{algorithmic} to find arithmetical structures on a square-free irreducible dominated polynomial with integer coefficients. 

If $f\in \mathbb{Z}[X]$ and all nonconstant coefficients of $f$ are positive, let 
\[
\mathcal{C}(f)=\{\mathbf{d}\in \mathbb{N}_+^n\, |\ f(X+\mathbf{d})\in \mathcal{D}_{\geq 0}(f) \}.
\] 

Now let $\min \mathcal{D}_{\geq 0}(f)$ be the set of all minimal elements of $\mathcal{D}_{\geq 0}(f)$. It is not difficult to check that $\min\mathcal{C}(f)$ exists and is finite by Dickson's Lemma. 
Also, for any $\mathbf{d}\in \mathbb{Z}^{n-1}$ and $1\leq s\leq n$, let $\mathbf{d}^{(s)}\in \mathbb{Z}^{n}$ be given by
\begin{equation} \label{expand}
(\mathbf{d}^{(s)})_i=
\begin{cases}
\mathbf{d}_i & \text{if } 1\leq i< s,\\
1 & \text{    if } i=s,\\
\mathbf{d}_{i-1} & \text{ if } s < i \leq n.
\end{cases}
\end{equation}

\begin{Algorithm}\label{A1Pol}
\mbox{}
\noindent\hrulefill

\textbf{Input}: An irreducible square-free dominated polynomial $f$ over $\mathbb{Z}$.

\textbf{Output}: $\min\mathcal{D}_{\geq 0}(f)$ and $\mathcal{D}(f)$.
\begin{enumerate}[\hspace{1pt}(1)]
\item \hspace{0mm} Let $\partial_s f=\dfrac{\partial f}{\partial x_s}$ for all $1 \leq s \leq n$.
\item \hspace{0mm} Compute $\tilde{A}_s=\min \mathcal{D}_{\geq 0}(\partial_s f)$ for all $1 \leq s \leq n$.
\vspace{0.7mm}
\item \hspace{0mm} Let $A_s=\{\mathbf{\tilde{d}}^{(s)}\,|\,\mathbf{\tilde{d}}\in \tilde{A}_s\}$.
\vspace{1mm}
\item \hspace{0mm} \textbf{For} $\boldsymbol{\delta}$ in $\prod_{s=1}^n A_s$:
\vspace{2.25mm}
\item \hspace{4mm} $\mathbf{d}=\sup\{\boldsymbol{\delta}_1,\boldsymbol{\delta}_2\ldots, \boldsymbol{\delta}_n \}$.
\vspace{0.8mm}
\item \hspace{4mm} Let $S=\{s\ |\ \mathrm{coef}_{\mathbf{d}}(x_s) = 0 \}$ 

\item \hspace{4mm} \textbf{If} $|S| = 0$:
\item \hspace{8mm} \textbf{For} $\mathbf{d^{*}} \in \min \mathcal{C}(\mathrm{Diag}(f(X+\mathbf{d}))$: 
\item \hspace{12mm} ``Add" $\mathbf{d^{*}}+\mathbf{d}$ to $\min\mathcal{D}_{\geq 0}(f)$.

\item \hspace{4mm} \textbf{If} $|S| \geq 1$:
\item \hspace{8mm}
\textbf{For} $t\notin S$:
\item \hspace{12mm} Make $\mathbf{d}_t^{'}=\mathbf{d}_t+1$, $\mathbf{d}^{'}_r =\mathbf{d}_r$ for all $r\in [n]\setminus \{t\}$ 
\item \hspace{12mm} \textbf{For} $\mathbf{d^{*}} \in \min \mathcal{C}(\mathrm{Diag}(f(X+\mathbf{d}^{'}))$: 
\item \hspace{16mm} ``Add" $\mathbf{d^{*}}+\mathbf{d}^{'}$ to $\min\mathcal{D}_{\geq 0}(f)$.
\item \hspace{4mm} \textbf{If} $|S|\geq 2$:
\item \hspace{8mm}
\textbf{For} $s_1,s_2\in S$ ($s_1\neq s_2$):
\item \hspace{12mm}
Make $\mathbf{d}^{'}_{s_1}=\mathbf{d}_{s_1}+1$, $\mathbf{d}^{'}_{s_2}=\mathbf{d}_{s_2}+1$, $\bff{d}^{'}_r=\bff{d}_r$ for all $r\in [n]\setminus \{s_1,s_2\}$ 
\item \hspace{12mm} \textbf{For} $\mathbf{d^{*}} \in \min \mathcal{C}(\mathrm{Diag}(f(X+\mathbf{d}^{'}))$: 
\item \hspace{16mm} ``Add" $\mathbf{d^{*}}+\mathbf{d}^{'}$ to $\min\mathcal{D}_{\geq 0}(f)$.

\item \hspace{0mm} Return $\min \mathcal{D}_{\geq 0}(f)$ and $\mathcal{D}(f)=\{ \bff{d}\in\min\mathcal{D}_{\geq 0}(f)\ |\ f(\mathbf{d})=0 \}$.
\end{enumerate}
\noindent\hrulefill 
\end{Algorithm}

The vector at step (5) is the supremum of the set of vectors $\{\boldsymbol{\delta}_1,\ldots , \boldsymbol{\delta}_n \}$ under the usual (entry by entry) order.
The function ``add" at steps (9), (14) and (19) means that we add the corresponding vector to the set $\min\mathcal{D}_{\geq 0}(L)$ whenever it is not greater than other vector already in the set.
Afterwards, by erasing every vector greater than said vector from the set, then the minimality of the set is assured.
The proof of the correctness of Algorithm~\ref{A1Pol} will be similar to the one given for~\cite[Algorithm 3.2]{algorithmic}. 
Thus we begin by extending~\cite[Lemma 3.1]{algorithmic} for the polynomial case.  

\begin{Lemma}\label{lemmapol}
If $a,b_1,b_2,c\in \mathbb{Z}$, $a\geq 1$ and $f=ax_1x_2+b_1x_1+b_2x_2+c$, then
\[
\min \mathcal{D}_{\geq 0}(f)=\min \left\{ \left(d,\max\Big(d_2^+,\Big\lceil \frac{-(c+b_1d)}{ad+b_2}\Big\rceil \Big)\right) \Big|\, d\in\mathbb{N}_+,\, d_1^{+} \leq d\leq \max \Big(d_1^+,\Big\lceil \frac{-(c+b_2d_2^{+} )}{ad_2^{+} +b_1 } \Big\rceil \Big) \right\},
\]
where $d_1^{+}=\max(1,\lceil\frac{1-b_2}{a}\rceil)$ and $d_2^{+}=\max(1,\lceil\frac{1-b_1}{a}\rceil)$.
\end{Lemma}
\begin{proof}
A vector $\bff{d}=(d_1,d_2)\in\mathbb{Z}^2$ is in $\mathcal{D}_{\geq 0}(f)$ if and only if 
$d_1,d_2\geq 1$, $ad_1+b_2,ad_2+b_1\geq 1$ and
\begin{equation}\label{basecondition}
ad_1d_2+b_1d_1+b_2d_2+c\geq 0.
\end{equation}
We set $d_1^+=\max(1,\lceil\frac{1-b_2}{a}\rceil)$ and $d_2^{+}=\max(1,\lceil\frac{1-b_1}{a}\rceil)$. 
It is clear that if $\bff{d}\in\mathcal{D}_{\geq 0}(f)$, then $\bff{d}\geq (d_1^+,d_2^+)$.
On the other hand, if $(d_1,d_2)\geq (d_1^+,d_2^+)$, then the only condition left for $\bff{d}$ to be in $\mathcal{D}_{\geq 0}(f)$ is \ref{basecondition}. 
Therefore, if $ad_1^+ d_2^+ +b_1d_1^+ +b_2d_2^+ +c\geq 0, \text{ then } \min \mathcal{D}_{\geq 0}(f)=\{(d_1^+,d_2
^+)\}$.
Henceforth, let us assume that
\begin{equation}\label{condition1}
ad_1^+ d_2^+ +b_1d_1^+ +b_2d_2^+ +c < 0\ (\leq -1)
\end{equation} 
and
\begin{equation}\label{condition2}
 ad_1d_2^+ +b_1d_1+b_2d_2^+ +c < 0.
\end{equation}
Thus $d_1^+\leq d_1 < \frac{-(c+b_2d_2^+ )}{ad_2^+ +b_1}$
and in order to fulfill condition (\ref{basecondition}), we have that $d_2\geq \frac{-(c+b_1d_1)}{ad +b_2}$. 
Also note that $\max(d_2^+, \frac{-(c+b_1d_1)}{ad_1 +b_2} )= \frac{-(c+b_1d_1)}{ad_1 +b_2}$ by (\ref{condition2}). 
Then 
\[
\min \left\{ (d_1,\lceil \frac{-(c+b_1d_1)}{ad_1 +b_2}\rceil )|\,d_1^+\leq d_1\leq \lfloor \frac{-(c+b_2d_2^+)}{ad_2^+ +b_1}\rfloor \right\} \subseteq \min\mathcal{D}_{\geq 0}(f).
\]
Finally, if 
\begin{equation}\label{condition3}
 ad_1d_2^+ +b_1d_1+b_2d_2^+ +c \geq 0,
\end{equation}
then we have that $\max(d_2^+, \frac{-(c+b_1d_1)}{ad_1 +b_2} )=d_2^+$ and
$
d_1\geq \frac{-(c+b_2d_2^+ )}{ad_2^+ +b_1}.
$
Thus 
\[
\min\{\bff{d}\in\mathcal{D}_{\geq 0}(f)|\,\text{ \ref{condition1} and \ref{condition3} holds} \}=\{(\lceil \frac{-(c+b_2d_2^+ )}{ad_2^+ +b_1}\rceil, d_2^+)\}.
\] 
We conclude that
\[
\min \mathcal{D}_{\geq 0}(f)=\begin{cases}
\min \left\{ \{ (d,\lceil\frac{-(c+b_1d)}{ad+b_2}\rceil )|\,d_1^+\leq d\leq \lfloor \frac{-(c+b_2d_2^+)}{ad_2^+ +b_1}\rfloor \}\cup\{ (\lceil\frac{-(c+b_2d_2^+)}{ad_2^+ +b_1}\rceil,d_2^+ )\} \right\} & \text{if \ref{condition1} holds},\\
\{(d_1^+,d_2^+)\} &\text{otherwise}. 
\end{cases}
\]
Clearly, this can be restated so that we have the result.
\end{proof}

\begin{Remark}
Note that $\mathcal{D}_{\geq 0}(f)$ is an infinite set, but by Dickson's Lemma $\min \mathcal{D}_{\geq 0}(f)$ is finite.
Also $f$ is monotone, that is, if $g(x_1,x_2)=f(x_1+d_1^+ ,x_2+d_2^+)$ has positive non-constant coefficients, then $g(x_1+\epsilon_1^{'},x_2+\epsilon_2^{'}) > g(x_1+\epsilon_1,x_2+\epsilon_2) > g(x_1,x_2)$ for every $(\epsilon_1^{'},\epsilon_2^{'}) > (\epsilon_1,\epsilon_2) > 0$. 
\end{Remark}

\begin{Example}\label{examplefig2}
Let $f=f(x_1,x_2)=2x_1x_2-7x_1-10x_2+16$ and let $d_1^+$ and $d_2^+$ be as in Lemma~\ref{lemmapol}. 
It is not difficult to check that $(d_1^+,d_2^+)=(6,4)$ and 
\begin{eqnarray*}
\min \mathcal{D}_{\geq 0}(f)&=&\min\left\{ \Big(d,\max\big(4,\Big\lceil\frac{-(16-7d)}{2d-10}\Big\rceil \big) \Big)\Big|\,d\in\mathbb{N}_+,\ 6\leq d\leq 24 \right\}\\
&=&\min\left\{ \begin{tabular}{c}
(6,13),(7,9),(8,7),(9,6),(10,6),(11,6),(12,5),(13,5),(14,5),(15,5),\\
(16,5),(17,5),(18,5),(19,5),(20,5),(21,5),(22,5),(23,5),(24,4)\end{tabular}
\right\}\\
&=&\left\{ \begin{tabular}{c}
 (6,13),(7,9),(8,7),(9,6),(12,5),(24,4) \end{tabular}\right\}.
\end{eqnarray*}
And therefore $\mathcal{D}(f)=\{(6,13),(24,4)\}$.
\end{Example}

Now we proceed to prove that the Algorithm~\ref{A1Pol} is correct.

\begin{Theorem}
Algorithm~\ref{A1Pol} computes the sets $\min\mathcal{D}_{\geq 0}(f)$ and $\mathcal{D}(f)$ for any irreducible square-free dominated polynomial $f\in\mathbb{Z}[X]$.
\end{Theorem}
\begin{proof}
First, without loss of generality we can assume that every variable in $X$ appears in some monomial of $f$ and that $|X|=n$.
In the case of a matrix $L$, induction on the size of $L$ and the $n-1$ minors of $(\textrm{Diag}(X+\mathbf{d})-L)$ correspond to induction on the degree of $f$ and its first partial derivatives respectively. 
Thus, we will proceed by induction on the number of variables in $X$, which is the degree of $f$.

If $f=f(X)$ is a square-free dominated polynomial with $|X|=2$ and positive leading coefficient we have that $X=\{x_1,x_2 \}$ and $f=ax_1x_2+b_1x_1+b_2x_2+c$ and therefore we get the result by Lemma~\ref{lemmapol}. 

Now, assume that the algorithm is correct for every number of variables up to $n-1$ and let $X=\{x_1,\ldots , x_n\}$ and $f\in \mathbb{Z}[X]$ be an irreducible square-free dominated polynomial of degree $n$ with positive leading coefficient.
It is not difficult to check, that steps (1) to (5) of Algorithm~\ref{A1Pol} creates a set of vectors $\Delta$ such that if $\bff{d}\in \Delta$, then the nonconstant coefficients of any monomial of degree at least 2 in $f(x_1+d_1,\ldots,x_n+d_n)$ are positive.
Moreover, the nonconstant coefficients of any term of degree one of $f(X+\mathbf{d})$ are non-negative whereas the constant term may be negative.

If $S=\emptyset$ implies that every nonconstant coefficient of $f(X+\mathbf{d})$ is positive, see Step (7).
Steps (10) - (12) and steps (15) - (17) handle the other two cases. 
That is, we have that all nonconstant coefficients of $f(X+\mathbf{d}^{'})$ are positive. 
Let $\Delta^{'}$ be the set of all of these vectors obtained at steps of Algorithm \ref{A1Pol}.
We will prove that in steps (8)-(9), (13)-(14), (18)-(19) and (20), the algorithm increases the vectors in $\Delta^{'}$ further so that we get all the vectors in $\min\mathcal{D}_{\geq 0} (f)$.
Note that if $\mathbf{d}^{'}\in \Delta^{'}$, then by the definition of $\mathcal{C}(f)$ every vector $\mathbf{u}\geq\mathbf{d}^{'}$ such that $f(X+\bff{u})$ has all of its noncontant coefficients positive and the constant non-negative coefficient can be reached on steps (8)-(9), (13)-(14) and (18)-(19). 
Therefore we only need to prove that every $\mathbf{u}\in\min\mathcal{D}_{\geq 0}(f)$ is reached by some vector in $\Delta^{'}$.

In order to prove this, for every $\mathbf{u}\in\mathcal{D}_{\geq 0}(f)$, let $\mathbf{u}_{|s}$ be the vector equal to $\mathbf{u}$ without the $s$-th entry.
That is,
\[
(\mathbf{u}_{|s})_i=\begin{cases}
\mathbf{u}_i, & \text{ if }1\leq i\leq s-1,\\
\mathbf{u}_{i+1}, & \text{ if }s\leq i\leq n-1.
\end{cases}
\]
Then for every $s\in[n]$, we have that $\mathbf{u}_{|s}\in\mathcal{D}_{\geq 0}(\partial_s f)$ and there exists $\mathbf{\tilde{u}}\in\min\mathcal{D}_{\geq 0}(\partial_s f)$ such that $\mathbf{\tilde{u}}\leq \mathbf{u}_{|s}$. 
Consequently, we have that 
\[
\max_{s\in [n]} \left\{ (\mathbf{\tilde{u}}^{(s)})_i \right\} \leq \mathbf{u}_i,
\]
where $\mathbf{u}^{(s)}$ is as in equation (\ref{expand}). 
In other words, every $\mathbf{u}\in\min\mathcal{D}_{\geq 0}(f)$ is greater or equal than a vector presented by step (5). 
Therefore let $\mathbf{u}\in\min\mathcal{D}_{\geq 0}(f)$ and let $\bff{d}\leq \bff{f}$ be such vector given at step (5). 
Then, assume that there is no vector $\bff{d}^{'}\geq \bff{d}$ in $\Delta^{'}$ such that $\bff{u}\geq \bff{d}^{'}$. 
Note that $S=\{s\ |\ \mathrm{coef}_{\mathbf{d}}(x_s) = 0 \}\neq \emptyset$ and that any vector in $\Delta^{'}$ can not be greater or equal than $\bff{u}$. 
Thus $\bff{u}=\bff{d}+a\bff{e_s}$ for some $a\in \mathbb{N}_+$ and some $s\in S$, where $\bff{e_s}\in \mathbb{N}^n$ is the standard unit vector with its $s$-th entry equal to $1$. 
Therefore $\textrm{coef}_{\bff{u}}(x_s)=0$, a contradiction since $\bff{u}\in \min\mathcal{D}_{\geq 0}(f)$. 
Concluding that there is a vector $\bff{d}^{'}\geq \bff{d}$ in $\Delta^{'}$ such that $\bff{u}\geq \bff{d}^{'}$ and therefore the algorithm computes $\min\mathcal{D}_{\geq 0}(f)$ and $\mathcal{D}(f)$. 
\end{proof}
The next example illustrates how Algorithm~\ref{A1Pol} works on a polynomial which is not the determinant of a matrix with variables in the diagonal. 

\begin{Example}
Let $f=x_1x_2x_3-19x_1+2x_2+3x_3-23$ be the irreducible polynomial given in Example~\ref{p23}.
Step $(2)$ of Algorithm~\ref{A1Pol} gives us
\[
\begin{array}{ccc}
\partial_1 f=x_2x_3-19 & \partial_2 f=x_1x_3+2 & \partial_3 f=x_1x_2+3.
\end{array}
\]
From step $(3)$ and Lemma \ref{lemmapol} we get that $\min \mathcal{D}_{\geq 0}(\partial_2 f)=\min \mathcal{D}_{\geq 0}(\partial_3 f)=\{(1,1)\}$ and
\[
\min \mathcal{D}_{\geq 0}(\partial_1 f)=\{ (1,19), (19,1), (2,10), (10,2), (3,7), (7,3), (4,5), (5,4) \}.
\]
Continuing with Algorithm \ref{A1Pol} we have the following set of vectors to search,
\[
\Pi=
\begin{Bmatrix}
(1,1,19) & (1,2,10) & (1,3,7) & (1,4,5)\\
(1,19,1) & (1,10,2) & (1,7,3)  & (1,5,4)
\end{Bmatrix}.
\]
Note that $f_{\mathbf{d}}(X)$ has positive constant term for almost every vector $\mathbf{d}\in\Pi$, except for $(1,5,4)$.
That is, only the vector $(1,5,4)$ has the chance to be an arithmetical structure of $f$.
Indeed, since
\[
f_{(1,5,4)}(X)=x_1x_2x_3+4x_1x_2+5x_1x_3+x_2x_3+x_1+6x_2+8x_3+0,
\]
then $\mathcal{D}(f)=\{(1,5,4)\}$.
\end{Example}

Next Figure illustrate the geometry of Lemma~\ref{lemmapol}.
If $P_G$ is the green region as $P_G$, then it corresponds to $\mathcal{D}_{\geq 0}(f)$ since it is the portion of the $\mathbb{N}_+$-grid ``above" $(d_1^+,d_2^+)$ and such that $f\geq 0$. 
More precisely, $\mathcal{D}_{\geq 0}(f)=P_G\cap\mathbb{N}_+^2$. 
Furthermore, it is not difficult to see that if $g$ is a polynomial of degree $n$ then $\mathcal{D}_{\geq 0}(g)= P\cap \mathbb{N}_+^n$, where $P$ is an unbounded n-dimensional polytope.
    
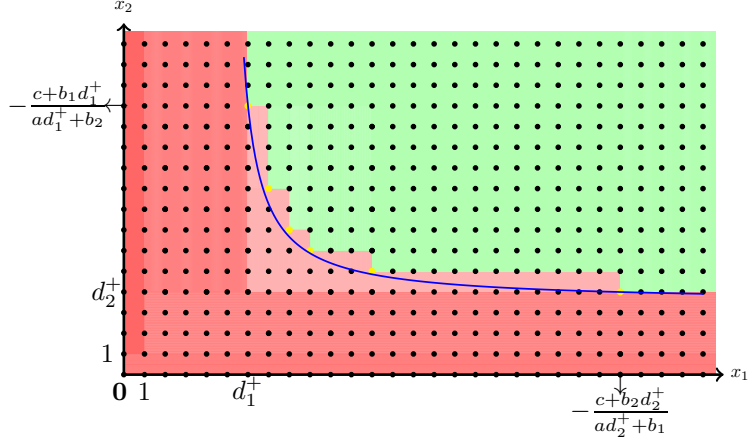
\begin{figure}[ht] \centering
\begin{tikzpicture}[line width=1pt, scale=0.55]
\tikzstyle{every node}=[inner sep=0pt, minimum width=1.1pt]
\shade[bottom color=red!60,top color=red!60] (0,0) rectangle +(14.3,0.5);
\shade[bottom color=red!60,top color=red!60] (0,0.5) rectangle +(0.5,7.8);

\shade[bottom color=red!50,top color=red!50] (0.5,0.5) rectangle +(13.8,1.5);
\shade[left color=red!50,right color=red!50] (0.5,2) rectangle +(2.5,6.3);
\shade[left color=red!30,right color=red!30] (3,2) rectangle +(9,0.5);
\shade[left color=red!30,right color=red!30] (3,2.5) rectangle +(3,0.5);
\shade[left color=red!30,right color=red!30] (3,3) rectangle +(1.5,0.5);
\shade[left color=red!30,right color=red!30] (3,3.5) rectangle +(1,1);
\shade[left color=red!30,right color=red!30] (3,4.5) rectangle +(0.5,2);

\shade[left color=green!30,right color=green!30] (6,2.5) rectangle +(6,4);
\shade[left color=green!30,right color=green!30] (4.5,3) rectangle +(1.5,3.5);
\shade[left color=green!30,right color=green!30] (4,3.5) rectangle +(0.5,3);
\shade[left color=green!30,right color=green!30] (3.5,4.5) rectangle +(0.5,2);

\shade[left color=green!30,right color=green!30] (12,2) rectangle +(2.3,4.5);
\shade[left color=green!30,right color=green!30] (3,6.5) rectangle +(9,1.8);

\shade[left color=green!30,right color=green!30] (12,6.5) rectangle +(2.3,1.8);

\draw (14,0)--(0,0)--(0,8);
\draw[->](0,8)->(0,8.5);
\draw[->] (14,0)->(14.5,0);
\node at (0,8.9)[auto,scale=0.7] {\small $x_2$};
\node at (14.9,0)[auto,scale=0.7] {\small$x_1$};

\filldraw[black] (2,3) circle (1pt);\filldraw[black] (4,3) circle (1pt);\filldraw[black] (5,3) circle (1pt);\filldraw[black] (1,4) circle (1pt);\filldraw[black] (3,3) circle (1pt);\filldraw[black] (1,4) circle (1pt);\filldraw[black] (2,4) circle (1pt);\filldraw[black] (3,4) circle (1pt);\filldraw[black] (4,4) circle (1pt);\filldraw[black] (5,4) circle (1pt);\filldraw[black] (1,5) circle (1pt);\filldraw[black] (2,5) circle (1pt);\filldraw[black] (3,5) circle (1pt);\filldraw[black] (4,5) circle (1pt);\filldraw[black] (5,5) circle (1pt);\filldraw[black] (3,2) circle (1pt);\filldraw[black] (4,2) circle (1pt);\filldraw[black] (5,2) circle (1pt);\filldraw[black] (4,1) circle (1pt);\filldraw[black] (5,1) circle (1pt);\filldraw[black] (5,1) circle (1pt);\filldraw[black] (6,0) circle (1pt);\filldraw[black] (6,1) circle (1pt);\filldraw[black] (6,2) circle (1pt);\filldraw[black] (6,3) circle (1pt);\filldraw[black] (6,4) circle (1pt);\filldraw[black] (6,5) circle (1pt);\filldraw[black] (6,6) circle (1pt);\filldraw[black] (0,6) circle (1pt);\filldraw[black] (1,6) circle (1pt);\filldraw[black] (2,6) circle (1pt);\filldraw[black] (3,6) circle (1pt);\filldraw[black] (4,6) circle (1pt);\filldraw[black] (5,6) circle (1pt);\filldraw[black] (0,4) circle (1pt);\filldraw[black] (0,5) circle (1pt);\filldraw[black] (2,2) circle (1pt);\filldraw[black] (4,0) circle (1pt);\filldraw[black] (4.5,0) circle (1pt);\filldraw[black] (5,0) circle (1pt);\filldraw[black] (5.5,0) circle (1pt);\filldraw[black] (6.5,0) circle (1pt);\filldraw[black] (7,0) circle (1pt);\filldraw[black] (7.5,0) circle (1pt);\filldraw[black] (8,0) circle (1pt);\filldraw[black] (8.5,0) circle (1pt);\filldraw[black] (9,0) circle (1pt);\filldraw[black] (9.5,0) circle (1pt);\filldraw[black] (10,0) circle (1pt);\filldraw[black] (10.5,0) circle (1pt);\filldraw[black] (11,0) circle (1pt);\filldraw[black] (11.5,0) circle (1pt);\filldraw[black] (12,0) circle (1pt);\filldraw[black] (12,0) circle (1pt);\filldraw[black] (12.5,0) circle (1pt);\filldraw[black] (13,0) circle (1pt);\filldraw[black] (13.5,0) circle (1pt);\filldraw[black] (14,0) circle (1pt);\filldraw[black] (0.5,0.5) circle (1pt);\filldraw[black] (1,0.5) circle (1pt);
\filldraw[black] (1.5,0.5) circle (1pt);\filldraw[black] (2,0.5) circle (1pt);\filldraw[black] (2.5,0.5) circle (1pt);\filldraw[black] (3,0.5) circle (1pt);\filldraw[black] (3.5,0.5) circle (1pt);\filldraw[black] (4,0.5) circle (1pt);\filldraw[black] (4.5,0.5) circle (1pt);\filldraw[black] (5,0.5) circle (1pt);\filldraw[black] (5.5,0.5) circle (1pt);\filldraw[black] (6,0.5) circle (1pt);\filldraw[black] (6.5,0.5) circle (1pt);\filldraw[black] (7,0.5) circle (1pt);\filldraw[black] (7.5,0.5) circle (1pt);\filldraw[black] (8,0.5) circle (1pt);\filldraw[black] (8.5,0.5) circle (1pt);\filldraw[black] (9,0.5) circle (1pt);\filldraw[black] (9.5,0.5) circle (1pt);\filldraw[black] (10,0.5) circle (1pt);\filldraw[black] (10.5,0.5) circle (1pt);\filldraw[black] (11,0.5) circle (1pt);\filldraw[black] (11.5,0.5) circle (1pt);\filldraw[black] (12,0.5) circle (1pt);\filldraw[black] (12,0.5) circle (1pt);\filldraw[black] (12.5,0.5) circle (1pt);\filldraw[black] (13,0.5) circle (1pt);\filldraw[black] (13.5,0.5) circle (1pt);\filldraw[black] (14,0.5) circle (1pt);\filldraw[black] (0.5,1.5) circle (1pt);\filldraw[black] (1,1.5) circle (1pt);
\filldraw[black] (1.5,1.5) circle (1pt);\filldraw[black] (2,1.5) circle (1pt);\filldraw[black] (2.5,1.5) circle (1pt);\filldraw[black] (3,1.5) circle (1pt);\filldraw[black] (3.5,1.5) circle (1pt);\filldraw[black] (4,1.5) circle (1pt);\filldraw[black] (4.5,1.5) circle (1pt);\filldraw[black] (5,1.5) circle (1pt);\filldraw[black] (5.5,1.5) circle (1pt);\filldraw[black] (6,1.5) circle (1pt);\filldraw[black] (6.5,1.5) circle (1pt);\filldraw[black] (7,1.5) circle (1pt);\filldraw[black] (7.5,1.5) circle (1pt);\filldraw[black] (8,1.5) circle (1pt);\filldraw[black] (8.5,1.5) circle (1pt);\filldraw[black] (9,1.5) circle (1pt);\filldraw[black] (9.5,1.5) circle (1pt);\filldraw[black] (10,1.5) circle (1pt);\filldraw[black] (10.5,1.5) circle (1pt);\filldraw[black] (11,1.5) circle (1pt);\filldraw[black] (11.5,1.5) circle (1pt);\filldraw[black] (12,1.5) circle (1pt);\filldraw[black] (12,1.5) circle (1pt);\filldraw[black] (12.5,1.5) circle (1pt);\filldraw[black] (13,1.5) circle (1pt);\filldraw[black] (13.5,1.5) circle (1pt);\filldraw[black] (14,1.5) circle (1pt);

\filldraw[black] (0.5,2) circle (1pt);\filldraw[black] (1,2) circle (1pt);\filldraw[black] (1.5,2) circle (1pt);\filldraw[black] (2,2) circle (1pt);\filldraw[black] (2.5,2) circle (1pt);\filldraw[black] (3,2) circle (1pt);\filldraw[black] (3.5,2) circle (1pt);\filldraw[black] (4,2) circle (1pt);\filldraw[black] (4.5,2) circle (1pt);\filldraw[black] (5,2) circle (1pt);\filldraw[black] (5.5,2) circle (1pt);\filldraw[black] (6,2) circle (1pt);\filldraw[black] (6.5,2) circle (1pt);\filldraw[black] (7,2) circle (1pt);\filldraw[black] (7.5,2) circle (1pt);\filldraw[black] (8,2) circle (1pt);\filldraw[black] (8.5,2) circle (1pt);\filldraw[black] (9,2) circle (1pt);\filldraw[black] (9.5,2) circle (1pt);\filldraw[black] (10,2) circle (1pt);\filldraw[black] (10.5,2) circle (1pt);\filldraw[black] (11,2) circle (1pt);\filldraw[black] (11.5,2) circle (1pt);\filldraw[yellow] (12,2) circle (1.5pt);\filldraw[black] (12.5,2) circle (1pt);\filldraw[black] (13,2) circle (1pt);\filldraw[black] (13.5,2) circle (1pt);\filldraw[black] (14,2) circle (1pt);

\filldraw[black] (0.5,2.5) circle (1pt);\filldraw[black] (1,2.5) circle (1pt);\filldraw[black] (1.5,2.5) circle (1pt);\filldraw[black] (2,2.5) circle (1pt);\filldraw[black] (2.5,2.5) circle (1pt);\filldraw[black] (3,2.5) circle (1pt);\filldraw[black] (3.5,2.5) circle (1pt);\filldraw[black] (4,2.5) circle (1pt);\filldraw[black] (4.5,2.5) circle (1pt);\filldraw[black] (5,2.5) circle (1pt);\filldraw[black] (5.5,2.5) circle (1pt);\filldraw[yellow] (6,2.5) circle (1.5pt);\filldraw[black] (6.5,2.5) circle (1pt);\filldraw[black] (7,2.5) circle (1pt);\filldraw[black] (7.5,2.5) circle (1pt);\filldraw[black] (8,2.5) circle (1pt);\filldraw[black] (8.5,2.5) circle (1pt);\filldraw[black] (9,2.5) circle (1pt);\filldraw[black] (9.5,2.5) circle (1pt);\filldraw[black] (10,2.5) circle (1pt);\filldraw[black] (10.5,2.5) circle (1pt);\filldraw[black] (11,2.5) circle (1pt);\filldraw[black] (11.5,2.5) circle (1pt);\filldraw[black] (12,2.5) circle (1pt);\filldraw[black] (12.5,2.5) circle (1pt);\filldraw[black] (13,2.5) circle (1pt);\filldraw[black] (13.5,2.5) circle (1pt);\filldraw[black] (14,2.5) circle (1pt);

\filldraw[black] (0.5,3) circle (1pt);\filldraw[black] (1,3) circle (1pt);\filldraw[black] (1.5,3) circle (1pt);\filldraw[black] (2,3) circle (1pt);\filldraw[black] (2.5,3) circle (1pt);\filldraw[black] (3,3) circle (1pt);\filldraw[black] (3.5,3) circle (1pt);\filldraw[black] (4,3) circle (1pt);\filldraw[yellow] (4.5,3) circle (1.5pt);\filldraw[black] (5,3) circle (1pt);\filldraw[black] (5.5,3) circle (1pt);\filldraw[black] (6,3) circle (1pt);\filldraw[black] (6.5,3) circle (1pt);\filldraw[black] (7,3) circle (1pt);\filldraw[black] (7.5,3) circle (1pt);\filldraw[black] (8,3) circle (1pt);\filldraw[black] (8.5,3) circle (1pt);\filldraw[black] (9,3) circle (1pt);\filldraw[black] (9.5,3) circle (1pt);\filldraw[black] (10,3) circle (1pt);\filldraw[black] (10.5,3) circle (1pt);\filldraw[black] (11,3) circle (1pt);\filldraw[black] (11.5,3) circle (1pt);\filldraw[black] (12,3) circle (1pt);\filldraw[black] (12.5,3) circle (1pt);\filldraw[black] (13,3) circle (1pt);\filldraw[black] (13.5,3) circle (1pt);\filldraw[black] (14,3) circle (1pt);

\filldraw[black] (0.5,3.5) circle (1pt);\filldraw[black] (1,3.5) circle (1pt);\filldraw[black] (1.5,3.5) circle (1pt);\filldraw[black] (2,3.5) circle (1pt);\filldraw[black] (2.5,3.5) circle (1pt);\filldraw[black] (3,3.5) circle (1pt);\filldraw[black] (3.5,3.5) circle (1pt);\filldraw[yellow] (4,3.5) circle (1.5pt);\filldraw[black] (4.5,3.5) circle (1pt);\filldraw[black] (5,3.5) circle (1pt);\filldraw[black] (5.5,3.5) circle (1pt);\filldraw[black] (6,3.5) circle (1pt);\filldraw[black] (6.5,3.5) circle (1pt);\filldraw[black] (7,3.5) circle (1pt);\filldraw[black] (7.5,3.5) circle (1pt);\filldraw[black] (8,3.5) circle (1pt);\filldraw[black] (8.5,3.5) circle (1pt);\filldraw[black] (9,3.5) circle (1pt);\filldraw[black] (9.5,3.5) circle (1pt);\filldraw[black] (10,3.5) circle (1pt);\filldraw[black] (10.5,3.5) circle (1pt);\filldraw[black] (11,3.5) circle (1pt);\filldraw[black] (11.5,3.5) circle (1pt);\filldraw[black] (12,3.5) circle (1pt);\filldraw[black] (12.5,3.5) circle (1pt);\filldraw[black] (13,3.5) circle (1pt);\filldraw[black] (13.5,3.5) circle (1pt);\filldraw[black] (14,3.5) circle (1pt);

\filldraw[black] (0.5,4) circle (1pt);\filldraw[black] (1,4) circle (1pt);\filldraw[black] (1.5,4) circle (1pt);\filldraw[black] (2,4) circle (1pt);\filldraw[black] (2.5,4) circle (1pt);\filldraw[black] (3,4) circle (1pt);\filldraw[black] (3.5,4) circle (1pt);\filldraw[black] (4,4) circle (1pt);\filldraw[black] (4.5,4) circle (1pt);\filldraw[black] (5,4) circle (1pt);\filldraw[black] (5.5,4) circle (1pt);\filldraw[black] (6,4) circle (1pt);\filldraw[black] (6.5,4) circle (1pt);\filldraw[black] (7,4) circle (1pt);\filldraw[black] (7.5,4) circle (1pt);\filldraw[black] (8,4) circle (1pt);\filldraw[black] (8.5,4) circle (1pt);\filldraw[black] (9,4) circle (1pt);\filldraw[black] (9.5,4) circle (1pt);\filldraw[black] (10,4) circle (1pt);\filldraw[black] (10.5,4) circle (1pt);\filldraw[black] (11,4) circle (1pt);\filldraw[black] (11.5,4) circle (1pt);\filldraw[black] (12,4) circle (1pt);\filldraw[black] (12.5,4) circle (1pt);\filldraw[black] (13,4) circle (1pt);\filldraw[black] (13.5,4) circle (1pt);\filldraw[black] (14,4) circle (1pt);

\filldraw[black] (0.5,4.5) circle (1pt);\filldraw[black] (1,4.5) circle (1pt);\filldraw[black] (1.5,4.5) circle (1pt);\filldraw[black] (2,4.5) circle (1pt);\filldraw[black] (2.5,4.5) circle (1pt);\filldraw[black] (3,4.5) circle (1pt);\filldraw[yellow] (3.5,4.5) circle (1.5pt);\filldraw[black] (4,4.5) circle (1pt);\filldraw[black] (4.5,4.5) circle (1pt);\filldraw[black] (5,4.5) circle (1pt);\filldraw[black] (5.5,4.5) circle (1pt);\filldraw[black] (6,4.5) circle (1pt);\filldraw[black] (6.5,4.5) circle (1pt);\filldraw[black] (7,4.5) circle (1pt);\filldraw[black] (7.5,4.5) circle (1pt);\filldraw[black] (8,4.5) circle (1pt);\filldraw[black] (8.5,4.5) circle (1pt);\filldraw[black] (9,4.5) circle (1pt);\filldraw[black] (9.5,4.5) circle (1pt);\filldraw[black] (10,4.5) circle (1pt);\filldraw[black] (10.5,4.5) circle (1pt);\filldraw[black] (11,4.5) circle (1pt);\filldraw[black] (11.5,4.5) circle (1pt);\filldraw[black] (12,4.5) circle (1pt);\filldraw[black] (12.5,4.5) circle (1pt);\filldraw[black] (13,4.5) circle (1pt);\filldraw[black] (13.5,4.5) circle (1pt);\filldraw[black] (14,4.5) circle (1pt);

\filldraw[black] (0.5,5) circle (1pt);\filldraw[black] (1,5) circle (1pt);\filldraw[black] (1.5,5) circle (1pt);\filldraw[black] (2,5) circle (1pt);\filldraw[black] (2.5,5) circle (1pt);\filldraw[black] (3,5) circle (1pt);\filldraw[black] (3.5,5) circle (1pt);\filldraw[black] (4,5) circle (1pt);\filldraw[black] (4.5,5) circle (1pt);\filldraw[black] (5,5) circle (1pt);\filldraw[black] (5.5,5) circle (1pt);\filldraw[black] (6,5) circle (1pt);\filldraw[black] (6.5,5) circle (1pt);\filldraw[black] (7,5) circle (1pt);\filldraw[black] (7.5,5) circle (1pt);\filldraw[black] (8,5) circle (1pt);\filldraw[black] (8.5,5) circle (1pt);\filldraw[black] (9,5) circle (1pt);\filldraw[black] (9.5,5) circle (1pt);\filldraw[black] (10,5) circle (1pt);\filldraw[black] (10.5,5) circle (1pt);\filldraw[black] (11,5) circle (1pt);\filldraw[black] (11.5,5) circle (1pt);\filldraw[black] (12,5) circle (1pt);\filldraw[black] (12.5,5) circle (1pt);\filldraw[black] (13,5) circle (1pt);\filldraw[black] (13.5,5) circle (1pt);\filldraw[black] (14,5) circle (1pt);

\filldraw[black] (0.5,5.5) circle (1pt);\filldraw[black] (1,5.5) circle (1pt);\filldraw[black] (1.5,5.5) circle (1pt);\filldraw[black] (2,5.5) circle (1pt);\filldraw[black] (2.5,5.5) circle (1pt);\filldraw[black] (3,5.5) circle (1pt);\filldraw[black] (3.5,5.5) circle (1pt);\filldraw[black] (4,5.5) circle (1pt);\filldraw[black] (4.5,5.5) circle (1pt);\filldraw[black] (5,5.5) circle (1pt);\filldraw[black] (5.5,5.5) circle (1pt);\filldraw[black] (6,5.5) circle (1pt);\filldraw[black] (6.5,5.5) circle (1pt);\filldraw[black] (7,5.5) circle (1pt);\filldraw[black] (7.5,5.5) circle (1pt);\filldraw[black] (8,5.5) circle (1pt);\filldraw[black] (8.5,5.5) circle (1pt);\filldraw[black] (9,5.5) circle (1pt);\filldraw[black] (9.5,5.5) circle (1pt);\filldraw[black] (10,5.5) circle (1pt);\filldraw[black] (10.5,5.5) circle (1pt);\filldraw[black] (11,5.5) circle (1pt);\filldraw[black] (11.5,5.5) circle (1pt);\filldraw[black] (12,5.5) circle (1pt);\filldraw[black] (12.5,5.5) circle (1pt);\filldraw[black] (13,5.5) circle (1pt);\filldraw[black] (13.5,5.5) circle (1pt);\filldraw[black] (14,5.5) circle (1pt);

\filldraw[black] (0.5,6) circle (1pt);\filldraw[black] (1,6) circle (1pt);\filldraw[black] (1.5,6) circle (1pt);\filldraw[black] (2,6) circle (1pt);\filldraw[black] (2.5,6) circle (1pt);\filldraw[black] (3,6) circle (1pt);\filldraw[black] (3.5,6) circle (1pt);\filldraw[black] (4,6) circle (1pt);\filldraw[black] (4.5,6) circle (1pt);\filldraw[black] (5,6) circle (1pt);\filldraw[black] (5.5,6) circle (1pt);\filldraw[black] (6,6) circle (1pt);\filldraw[black] (6.5,6) circle (1pt);\filldraw[black] (7,6) circle (1pt);\filldraw[black] (7.5,6) circle (1pt);\filldraw[black] (8,6) circle (1pt);\filldraw[black] (8.5,6) circle (1pt);\filldraw[black] (9,6) circle (1pt);\filldraw[black] (9.5,6) circle (1pt);\filldraw[black] (10,6) circle (1pt);\filldraw[black] (10.5,6) circle (1pt);\filldraw[black] (11,6) circle (1pt);\filldraw[black] (11.5,6) circle (1pt);\filldraw[black] (12,6) circle (1pt);\filldraw[black] (12.5,6) circle (1pt);\filldraw[black] (13,6) circle (1pt);\filldraw[black] (13.5,6) circle (1pt);\filldraw[black] (14,6) circle (1pt);

\filldraw[black] (0.5,6.5) circle (1pt);\filldraw[black] (1,6.5) circle (1pt);\filldraw[black] (1.5,6.5) circle (1pt);\filldraw[black] (2,6.5) circle (1pt);\filldraw[black] (2.5,6.5) circle (1pt);\filldraw[yellow] (3,6.5) circle (1.5pt);\filldraw[black] (3.5,6.5) circle (1pt);\filldraw[black] (4,6.5) circle (1pt);\filldraw[black] (4.5,6.5) circle (1pt);\filldraw[black] (5,6.5) circle (1pt);\filldraw[black] (5.5,6.5) circle (1pt);\filldraw[black] (6,6.5) circle (1pt);\filldraw[black] (6.5,6.5) circle (1pt);\filldraw[black] (7,6.5) circle (1pt);\filldraw[black] (7.5,6.5) circle (1pt);\filldraw[black] (8,6.5) circle (1pt);\filldraw[black] (8.5,6.5) circle (1pt);\filldraw[black] (9,6.5) circle (1pt);\filldraw[black] (9.5,6.5) circle (1pt);\filldraw[black] (10,6.5) circle (1pt);\filldraw[black] (10.5,6.5) circle (1pt);\filldraw[black] (11,6.5) circle (1pt);\filldraw[black] (11.5,6.5) circle (1pt);\filldraw[black] (12,6.5) circle (1pt);\filldraw[black] (12.5,6.5) circle (1pt);\filldraw[black] (13,6.5) circle (1pt);\filldraw[black] (13.5,6.5) circle (1pt);\filldraw[black] (14,6.5) circle (1pt);

\filldraw[black] (0.5,7) circle (1pt);\filldraw[black] (1,7) circle (1pt);\filldraw[black] (1.5,7) circle (1pt);\filldraw[black] (2,7) circle (1pt);\filldraw[black] (2.5,7) circle (1pt);\filldraw[black] (3,7) circle (1pt);\filldraw[black] (3.5,7) circle (1pt);\filldraw[black] (4,7) circle (1pt);\filldraw[black] (4.5,7) circle (1pt);\filldraw[black] (5,7) circle (1pt);\filldraw[black] (5.5,7) circle (1pt);\filldraw[black] (6,7) circle (1pt);\filldraw[black] (6.5,7) circle (1pt);\filldraw[black] (7,7) circle (1pt);\filldraw[black] (7.5,7) circle (1pt);\filldraw[black] (8,7) circle (1pt);\filldraw[black] (8.5,7) circle (1pt);\filldraw[black] (9,7) circle (1pt);\filldraw[black] (9.5,7) circle (1pt);\filldraw[black] (10,7) circle (1pt);\filldraw[black] (10.5,7) circle (1pt);\filldraw[black] (11,7) circle (1pt);\filldraw[black] (11.5,7) circle (1pt);\filldraw[black] (12,7) circle (1pt);\filldraw[black] (12.5,7) circle (1pt);\filldraw[black] (13,7) circle (1pt);\filldraw[black] (13.5,7) circle (1pt);\filldraw[black] (14,7) circle (1pt);

\filldraw[black] (0.5,7.5) circle (1pt);\filldraw[black] (1,7.5) circle (1pt);\filldraw[black] (1.5,7.5) circle (1pt);\filldraw[black] (2,7.5) circle (1pt);\filldraw[black] (2.5,7.5) circle (1pt);\filldraw[black] (3,7.5) circle (1pt);\filldraw[black] (3.5,7.5) circle (1pt);\filldraw[black] (4,7.5) circle (1pt);\filldraw[black] (4.5,7.5) circle (1pt);\filldraw[black] (5,7.5) circle (1pt);\filldraw[black] (5.5,7.5) circle (1pt);\filldraw[black] (6,7.5) circle (1pt);\filldraw[black] (6.5,7.5) circle (1pt);\filldraw[black] (7,7.5) circle (1pt);\filldraw[black] (7.5,7.5) circle (1pt);\filldraw[black] (8,7.5) circle (1pt);\filldraw[black] (8.5,7.5) circle (1pt);\filldraw[black] (9,7.5) circle (1pt);\filldraw[black] (9.5,7.5) circle (1pt);\filldraw[black] (10,7.5) circle (1pt);\filldraw[black] (10.5,7.5) circle (1pt);\filldraw[black] (11,7.5) circle (1pt);\filldraw[black] (11.5,7.5) circle (1pt);\filldraw[black] (12,7.5) circle (1pt);\filldraw[black] (12.5,7.5) circle (1pt);\filldraw[black] (13,7.5) circle (1pt);\filldraw[black] (13.5,7.5) circle (1pt);\filldraw[black] (14,7.5) circle (1pt);

\filldraw[black] (0.5,8) circle (1pt);\filldraw[black] (1,8) circle (1pt);\filldraw[black] (1.5,8) circle (1pt);\filldraw[black] (2,8) circle (1pt);\filldraw[black] (2.5,8) circle (1pt);\filldraw[black] (3,8) circle (1pt);\filldraw[black] (3.5,8) circle (1pt);\filldraw[black] (4,8) circle (1pt);\filldraw[black] (4.5,8) circle (1pt);\filldraw[black] (5,8) circle (1pt);\filldraw[black] (5.5,8) circle (1pt);\filldraw[black] (6,8) circle (1pt);\filldraw[black] (6.5,8) circle (1pt);\filldraw[black] (7,8) circle (1pt);\filldraw[black] (7.5,8) circle (1pt);\filldraw[black] (8,8) circle (1pt);\filldraw[black] (8.5,8) circle (1pt);\filldraw[black] (9,8) circle (1pt);\filldraw[black] (9.5,8) circle (1pt);\filldraw[black] (10,8) circle (1pt);\filldraw[black] (10.5,8) circle (1pt);\filldraw[black] (11,8) circle (1pt);\filldraw[black] (11.5,8) circle (1pt);\filldraw[black] (12,8) circle (1pt);\filldraw[black] (12.5,8) circle (1pt);\filldraw[black] (13,8) circle (1pt);\filldraw[black] (13.5,8) circle (1pt);\filldraw[black] (14,8) circle (1pt);

\filldraw[black] (0.5,1) circle (1pt);\filldraw[black] (1.5,1) circle (1pt);\filldraw[black] (2.5,1) circle (1pt);\filldraw[black] (3.5,1) circle (1pt);\filldraw[black] (3.5,1) circle (1pt);\filldraw[black] (4.5,1) circle (1pt);\filldraw[black] (5.5,1) circle (1pt);\filldraw[black];\filldraw[black] (6.5,1) circle (1pt);\filldraw[black] (7,1) circle (1pt);\filldraw[black] (7.5,1) circle (1pt);\filldraw[black] (8,1) circle (1pt);\filldraw[black] (8.5,1) circle (1pt);\filldraw[black] (9,1) circle (1pt);\filldraw[black] (9.5,1) circle (1pt);\filldraw[black] (10,1) circle (1pt);\filldraw[black] (10.5,1) circle (1pt);\filldraw[black] (11,1) circle (1pt);\filldraw[black] (11.5,1) circle (1pt);\filldraw[black] (12,1) circle (1pt);\filldraw[black] (12.5,1) circle (1pt);\filldraw[black] (13,1) circle (1pt);\filldraw[black] (13.5,1) circle (1pt);\filldraw[black] (14,1) circle (1pt);

\filldraw[black] (0,0) circle (1pt);
\filldraw[black] (0,0.5) circle (1pt);
\filldraw[black] (0,1.5) circle (1pt);
\filldraw[black] (0,2.5) circle (1pt);
\filldraw[black] (0,3.5) circle (1pt);
\filldraw[black] (0,4.5) circle (1pt);
\filldraw[black] (0,5.5) circle (1pt);
\filldraw[black] (0,6.5) circle (1pt);
\filldraw[black] (0,7) circle (1pt);
\filldraw[black] (0,7.5) circle (1pt);
\filldraw[black] (0,8) circle (1pt);

\filldraw[black] (0.5,0) circle (1pt);
\filldraw[black] (1,0) circle (1pt);
\filldraw[black] (1.5,0) circle (1pt);
\filldraw[black] (2,0) circle (1pt);
\filldraw[black] (2.5,0) circle (1pt);
\filldraw[black] (3,0) circle (1pt);
\filldraw[black] (3.5,0) circle (1pt);

\filldraw[black] (0,1) circle (1pt);
\filldraw[black] (0,2) circle (1pt);
\filldraw[black] (0,3) circle (1pt);
\filldraw[black] (1,1) circle (1pt);
\filldraw[black] (2,1) circle (1pt);
\filldraw[black] (3,1) circle (1pt);
\filldraw[black] (1,2) circle (1pt);
\filldraw[black] (1,3) circle (1pt);

\node at (-0.4,0.5)[auto,scale=1] {\small $1$};
\node at (-0.4,2)[auto,scale=1] {\small $d_2^+$};
\node at (-0.2,6.48)[auto,scale=1]{\small $\leftarrow$};
\node at (-1.6,6.5)[auto,scale=1]{\small $-\frac{c+b_1d_1^+}{ad_1^+ +b_2}$};

\node at (-0.1,-0.4)[auto,scale=1] {\small $\bff{0}$};
\node at (0.5,-0.4)[auto,scale=1] {\small $1$};
\node at (3,-0.4)[auto,scale=1] {\small $d_1^+$};
\node at (12,-0.2)[auto,scale=1]{\small $\downarrow$};
\node at (12,-0.9)[auto,scale=1]{\small $-\frac{c+b_2d_2^+}{ad_2^+ +b_1}$};

\begin{axis}
  [ scale=1.5,
    axis lines= none,
    ymin=0, ymax=16,
    xmin=0, xmax=28,
    xtick={0,6,24,28},
    ytick={0,4,13,16},
    width=109.2mm,height=69.1mm,
    xlabel=$x$,
    ylabel={$f(x) = x^2 - x +4$}
  ]
  \addplot [
    line width=1.2pt,
    domain=5.8:28, 
    samples=200, 
    color=blue,
    ]
    {(7*x - 16) / (2*x - 10)}; 
  \end{axis}
\end{tikzpicture}
\label{fig2}
\caption{The blue line represents the curve $f=2x_1x_2-7x_1-10x_1+16=0$ for $x_1 \geq 5.8$ and the yellow points are the elements in $\min\mathcal{D}_{\geq 0}(f)$.}
\end{figure}
We recall that if $f$ is a square-free dominated polynomial without any arithmetical structure, then this does not implies that $f=0$ has not integer solutions. 

\begin{Example}\label{notallsoltns}
Let $g=x_1x_2+17x_1-12x_2+27$. 
By Lemma \ref{lemmapol} we have that
\[
\min\mathcal{D}_{\geq 0}(g)=\{(13,1) \}.
\]
On the other hand, since $g(13,1)=249$, then $\mathcal{D}(g)=\emptyset$.
Nevertheless $g=0$ has sixteen different solutions in $\mathbb{Z}^2$. 
Moreover four of them are solutions in $\mathbb{N}_+^2$, namely
\[
\{(1,4),(5,16),(9,60),(11,214) \}.
\]
None of them found by the algorithm. 
Because the condition of having all non-constant coefficients positive is not fulfilled by any of them. 
For instance note that $f(x_1+11,x_2+214)=x_1x_2+231x_1-x_2$.
\end{Example}

Defining an $r$-arithmetical structure of an integer square-free dominated polynomial is a more difficult task.
For one hand, the $r$-arithmetical structures on $L$ and $L^t$ are equal if and only if $L$ is symmetric.
And for the other hand, $f_L(X)=f_{L^t}(X)$ for any $L\in M_n(\mathbb{Z})$ because the determinant of a matrix is invariant under the transpose, that is, $\det(L)=\det(L^t)$.
Moreover, if $M$ is a matrix without rows or columns equal to zero, then $\mathcal{D}(L)=\mathcal{D}(L^t)$.
That is, the polynomial $f_L(X)$ does not distinguish between $L$ and $L^t$.
However $r$-arithmetical structures on $L$ and $L^t$ are not equal when $L$ is not symmetric.
Therefore in general we may not try to extract the information of the $r$-arithmetical structures from $f_L(X)$.
Next example illustrate previous discussion.
 
\begin{Example}{\label{transpose}}
If $L=\begin{pmatrix}
0 & 1\\
3 & 0
\end{pmatrix}$, then $f_L(x_1,x_2)=f_{L^t}(x_1,x_2)=x_1x_2-3$ and therefore 
\[
\mathcal{A}(L)=\{ ((1,3),(1,1)),((3,1),(1,3)) \} \text{ and } \mathcal{A}(L^t)=\{ ((1,3),(3,1)),((3,1),(1,1)) \}.
\]
Thus $\mathcal{D}(f_L)=\{ (1,3),(3,1) \}=\mathcal{D}(f_{L^t})$ and $\mathcal{R}(f_L)=\{ (1,1),(1,3)\}\neq \{ (1,1),(3,1)\}=\mathcal{R}(f_{L^t})$.
\end{Example}

\begin{Remark}
If $f$ is an irreducible polynomial which is the determinant of a matrix with variables in the diagonal irreducible, then it comes from an irreducible matrix.
\end{Remark}

Since a symmetric $Z$-matrix $M$ is an almost non-singular $M$-matrix with $\det(M)=0$ if and only if there exists $\bff{r}>0$ such that 
\[ 
Adj(M)=\,|K(M)|\,\bff{r}^t \bff{r} > \bff{0},
\]
where $\ker_{\mathbb{Q}}(M)=\langle \bff{r} \rangle$ and $K(M)$ is the critical group of $M$, see \cite[Proposition 3.4]{arithmetical}.
Then is factible to define the critical group of a $d$-arithmetical structure of a polynomial $f$ as
\[
|K(f,\bff{d})|=\gcd(\mathrm{coef}_{f_{\mathbf{d}}(X)}(x_1),\ldots,\mathrm{coef}_{f_{\mathbf{d}}(X)}(x_n)).
\]

Given any non-negative matrix with zero diagonal $L$ such that every of its rows are different from $\mathbf{0}$, then $(L\mathbf{1},\mathbf{1})$ is the canonical arithmetical structure of $L$. 
In general for polynomials in $\mathbb{Z}[X]$ we can not recover the concept of canonical arithmetical structure. 
Furthermore, some polynomials are extremal in the sense that they have very few arithmetical structures. 
We illustrate this idea at the next example.

\begin{Example}\label{p23}
If $g=x_1x_2x_3-19x_1+2x_2+3x_3+b$, then
\[
b=\frac{-114}{n}-n\text{ where } n\in\mathrm{Div}(114)=\pm\{1,2,3,6,19,38,57,114 \}.
\]
Which implies that $b\in \pm\{25,41,59,115 \}$.
It is not difficult to check by Proposition~\cite[Proposition 3.7]{algorithmic} that $f(x_1,x_2,x_3)=x_1x_2x_3-19x_1+2x_2+3x_3-23$ is not the determinant of a matrix with variables in the diagonal.
Evaluating, it is easy to see that $(d_1,d_2,d_3)\in \mathbb{N}_+^3$ is an arithmetical structure of $f$ if and only if
\[
d_2d_3-19\geq 1 \text{ and }(d_2d_3-19)d_1+2d_2+3d_3=23.
\]
Thus we have that $\mathcal{D}(f)=\{ (1,5,4) \}$.
A follow up problem would be to study this type of polynomials, where we have a single d-arithmetical structure.
\end{Example}




\end{document}